\documentclass[ECP]{ejpecp} 

\SHORTTITLE{Spine of Fleming-Viot process}

\TITLE{On the spine of the two-particle Fleming-Viot process driven by Brownian motion
\support{Supported in part by Simons Foundation Grant 506732.}}

\AUTHORS{%
  Krzysztof~Burdzy\footnote{University of Washington, United States of America.
    \EMAIL{burdzy@uw.edu}}
  \and 
  Tvrtko Tadi\'c\footnote{Microsoft Corporation,
    United States of America. \EMAIL{tvrtko@math.hr}}}

\KEYWORDS{Fleming-Viot type process ; Brownian motion ; Bessel process ; logarithmic transformation} 

\AMSSUBJ{60G17 ; 60J65 ; 60J80} 

\SUBMITTED{November 16, 2021} 
\ACCEPTED{???} 

\VOLUME{0}
\YEAR{2020}
\PAPERNUM{0}
\DOI{10.1214/YY-TN}

\ABSTRACT{The spine of the two-particle  Fleming-Viot process driven by Brownian motion is not a Bessel-3 process.}

\newcommand{\E}{\operatorname{\mathds{E}}} 
\renewcommand{\P}{\operatorname{\mathds{P}}} 

\newcommand{\prt}{\partial}

\newcommand{\cV}{\mathcal{V}}
\newcommand{\C}{\mathds{C}}

\begin{document}



\section{Introduction}\label{section:introduction}

We start with an informal outline of the main idea of this note. A more detailed review, including history and citations, will be presented later in the introduction.

A Fleming-Viot process is a process with a branching structure (but not a branching process according to terminology adopted in the literature on branching processes). Under very mild assumptions, it has a unique spine. When the number of individuals in the population is very large, the distribution of the spine is expected to be very close to the distribution of the driving process conditioned on  survival forever. There is an example showing that the distribution of the spine may be different from the distribution of the driving process conditioned on  survival forever. The published example is rather artificial so we present in this note a different example illustrating the same claim. Our new example is more natural in the sense that it is based on a model examined in a number of papers on Fleming-Viot processes.

\subsection{Literature review}
Fleming-Viot-type processes were originally defined in \cite{BHM}. 
In this model, there is a population of fixed size. Every individual moves independently from all other individuals according to the same Markovian transition mechanism, in a domain with a boundary. When an individual hits the boundary, the indvidual is killed and an individual chosen randomly (uniformly)  from the survivors  
splits into two individuals and the process continues in this manner. The question of whether the process can be continued for all times was addressed in \cite{BHM,extinctionOfFlemingViot,BBF,GK}. 
All of these papers studied, among other processes,
Fleming-Viot processes driven by Brownian motion.

A very special case of a Fleming-Viot process is when there are only two individuals driven by Brownian motion on $[0,\infty)$ and $0$ plays that role of the boundary; this model was studied in \cite{extinctionOfFlemingViot,GK}. 
In particular, it was shown in both papers that this process has  infinite lifetime.

Every Fleming-Viot process has a unique spine, i.e., a trajectory inside the branching tree that never hits the boundary of the domain where the process is confined; this was proved under strong assumptions in \cite[Thm. 4]{GK} and later in the full generality in \cite{BB18}. 

It was proved in \cite{BB18} that if the state space is finite and the number of individuals in the population goes to infinity then the distributions of spine processes converge to the distribution of the driving  Markov process conditioned on  survival forever. The same result has also been proven for diffusions reflected normally off the boundary
of a compact domain with soft killing (see \cite{ToughArxiv}), and 
for Brownian motions on a Lipshitz domain with hard killing (see \cite{BurEng}). The proof given in  \cite{ToughArxiv} relies only on the driving process being a
Feller process on a compact domain converging exponentially quickly to its quasi-stationary distribution. Therefore convergence of the spine is better understood when the domain is compact, or
at least bounded, than when the domain is unbounded (as with the Fleming-Viot process
studied in this paper).

The rate of convergence of the distribution of the Fleming-Viot process driven by a general Markov process to the quasi-stationary distribution was investigated in \cite{Vill}.

In \cite{BB18},
an example was given of a Fleming-Viot process driven by a Markov process on a three-element state space such that one of the elements plays the role of the boundary, the population consists of two individuals and the distribution of the spine is not equal 
to the distribution of the driving  Markov process conditioned on  survival forever. A Markov process with a three-element state space seems to be a rather artificial example in the context of Fleming-Viot models. We will show that the spine of the Fleming-Viot process with two individuals driven by Brownian motions on $[0,\infty)$ has a spine with a distribution different from the distribution of Brownian motion conditioned to stay positive, i.e., the distribution of the 3-dimensional Bessel process.
The point of this note is to show that proving that the spine does not have the distribution of the 3-dimensional Bessel process is somewhat tricky. In hindsight, this does not seems to be difficult because our proof is quite elementary. Nevertheless, our previous attempts in \cite{IEDPaper,ldm2020} generated some new results but failed to show the difference.
In a sense that will be made more
precise later on in the paper, the spine is quite ``close'' to a
3-dimensional Bessel process and therefore it is quite hard to
distinguish the two. The problem has been open for some time and while the solution is not
really difficult we now have a  better understanding as
to why the spine is nevertheless similar to a 3-dimensional Bessel process
in certain respects.

\section{Model and main result}

We will now define a Fleming-Viot process and  other elements of the model.
Informally, the process consists of two independent Brownian particles starting at the same point in $(0,\infty)$. At the time when one of them hits 0, it is killed and the other one branches into two particles. The new particles start moving as independent Brownian motions and the scheme is repeated.

\subsection{Notation and definitions}
On the formal side, let
$(W_1(t):t\geq 0)$ and $(W_2(t):t\geq 0)$ be two independent Brownian motions starting from $W_1(0) =W_2(0)=1$. Let
\begin{align*}
T_0&=0,\\
Y_0&=1,\\
\tau_j &= \inf\{t\geq 0: W_j(t) =0\}, \qquad j=1,2,\\
T_1&=\min(\tau_1,\tau_2),\\
Y_1&=\max(W_1(T_1),W_2(T_1)),
\end{align*}
and for $k\geq 2$,
\begin{align*}
T_{k}&=\inf\{t>T_{k-1} : \min(W_1(t)-W_1(T_{k-1})+Y_{k-1},W_2(t)-W_2(T_{k-1})+Y_{k-1}) =0 \},\\
Y_k&=\max(W_1(T_k)-W_1(T_{k-1})+Y_{k-1},W_2(T_k)-W_2(T_{k-1})+Y_{k-1}).
\end{align*}
It follows from \cite[Thm.~5.4]{BBF} or \cite[Thm.~1]{GK} that, a.s.,
\begin{equation}
 T_k\to \infty. \label{eq:TInf}
\end{equation}
Hence, for any $t\geq 0$ we can find $j$ such that $t\in [T_{j-1},T_{j})$. Then we set
\begin{align}\label{j18.1}
\cV(t)&=(V_1(t),V_2(t))
      =(W_1(t)-W_1(T_{j-1})+Y_{j-1},W_2(t)-W_2(T_{j-1})+Y_{j-1}).
\end{align}
This completes the definition of $\{\cV(t), t\geq 0\}$, an example of a Fleming-Viot process. 

Let $J_t = J(t)$ denote the spine, i.e., $J_t = V_1(t)$ for
$t\in [T_{k-1},T_{k})$ 
if $V_1(T_k-) > V_2(T_k-)=0$.
If the last condition fails, we let  $J_t = V_2(t)$ for
$t\in [T_{k-1},T_{k})$.

Note that $J(T_k) = Y_k$ 
for all $k\geq 1$. 

Recall that $d$-dimensional Bessel process $X_t$ is defined by
\begin{align}\label{s26.3}
dX_t = dB_t + \frac {d-1}{2 X_t} dt,
\end{align}
where $B$ is Brownian motion; see \cite[Sect. 3.3 C]{KS}. It is well known that Brownian motion on $[0,\infty)$ conditioned to never hit 0 has the transition probabilities of 3-dimensional Bessel process; this theorem was first proved in \cite{Doob57}.

\subsection{Main result}

\begin{theorem}\label{s26.1}
The distributions of $\{J_t, 0\leq t <\infty\}$ and 3-dimensional Bessel process $\{X_t, 0\leq t <\infty\}$ starting from 1 are singular with respect to each other.
\end{theorem}

We will  explicitly define an event that has a strictly positive probability according to  the first distribution but not according to the second one, and vice versa.

\medskip
We will now review two attempts to prove Theorem \ref{s26.1} that failed.

The following version of the Law of the Iterated Logarithm was proved in \cite{ldm2020}.
\begin{theorem}\label{thm:LawOfTheIteratedLogartihm}
Almost surely, 
\begin{align}\label{eq:lil}
\limsup_{n\to\infty}\frac{Y_n}{\sqrt{2 T_n\log \log T_n}}=1.
\end{align}
\end{theorem}

The Law of Iterated Logarithm stated in \eqref{eq:lil} is the same
 as that for the 3-dimensional Bessel process (see \cite{SW}), which has the same distribution as the  one-dimensional Brownian motion conditioned not to hit 0. Hence, 
Theorem \ref{thm:LawOfTheIteratedLogartihm} does not eliminate the possibility that the spine $J_t$ has the distribution of Brownian motion conditioned not to hit 0.

In this note, we will prove the following result.
\begin{theorem}\label{s26.2}
For every $u>0$,
\begin{align*}
\frac{1}{u}\log\P\left(\inf_{s\geq 0} \log X(s)<-u \right)=-1=\lim_{t\to \infty}\frac{1}{t}\log\P\left(\inf_{n\geq 0} \log J(T_n)<-t \right).
\end{align*}
\end{theorem}

The general message from Theorems \ref{thm:LawOfTheIteratedLogartihm} and \ref{s26.2} is that it is hard to distinguish between the spine and 3-dimensional Bessel process by studying ``extreme'' behavior of the two processes. 
The proof of Theorem \ref{s26.1} will be based on the analysis of the processes on the ``logarithmic scale.''

\section{Bessel processes}\label{sec:bes}

Let 
\begin{align}\label{m29.13}
\rho(t)= \int_0^t \frac 1 {X_s^{2}}ds.
\end{align}

\begin{lemma}\label{m29.9}
If $X$ is three-dimensional Bessel process with $X_0=1$ and $B$ is Brownian motion with $B_0=0$ then $\{ \log X_{\rho^{-1}(t)}, t\geq 0\}$ has the same distribution as the process $\left\{B_t + \frac 1 2 t, t\geq 0\right\}$.
\end{lemma}

\begin{proof}

Recall the stochastic differential equation \eqref{s26.3} defining Bessel processes.
Let $f(x) = \log x$. Then $f'(x) = 1/x$ and $f''(x) = - 1/x^2$. Let $A_t = f(X_t)$.
Then by the Ito formula
\begin{align*}
dA_t &= df(X_t)= \frac1 X_t dB_t + \left( \frac {d-1} {2X_t} \cdot \frac{1}{X_t} - \frac {1}{ 2} \cdot \frac 1 {X_t^2}\right) dt
=  \frac{1}{X_t} dB_t +  \frac {d-2} {2X_t^2}  dt\\
&= e^{-A_t} dB_t + \frac {d-2} {2} e^{-2A_t}  dt.
\end{align*}
For 3-dimensional Bessel process, i.e., when $d=3$, the formula is
$$
dA_t =  e^{-A_t} dB_t + \frac{1}{2} e^{-2A_t}  dt.
$$
We see that
the process $A_t$ is a time change of 
the process $ B_t + \frac 1 2 t$, if we use the clock 
\begin{align*}
\rho(t)=\int_0^t e^{-2A_s}ds= \int_0^t \frac 1 {X_s^{2}}ds.
\end{align*}
In other words, $\{A_{\rho^{-1}(t)}, t\geq 0\} =\{ \log X_{\rho^{-1}(t)}, t\geq 0\}$ has the same distribution as the process $\left\{B_t + \frac 1 2 t, t\geq 0\right\}$.
\end{proof}

\begin{lemma}\label{cor:logBessMin}
Suppose that $X= \{X(t)\ :\ t\geq 0\}$ is the 3-dimensional Bessel process with $X(0)=1$.
Let $M=\inf_{t\geq 0} \log X(t)$.
Then $-M$ has the exponential distribution with mean 1.
\end{lemma}
\begin{proof}
Time change does not  affect the distribution of the infimum of a process, hence, by Lemma \ref{m29.9}, $M$ has the same distribution as 
$\min_{t\geq 0}\left( B_t +\frac{1}{2}t\right)$.
According to \cite[Sect. 3.3, Exercise 5.9]{KS}, $-M$ is exponential  with mean 1.
\end{proof}

\section{Logarithmic transformation of Fleming-Viot process}\label{section:logTransform}

We will use complex representation $V_1(t)+iV_2(t)$ of the process $\cV(t) = (V_1(t), V_2(t)) $ defined in \eqref{j18.1}. 
We apply the complex mapping $z\mapsto \log z$ to this process so that it is transformed into a process in the strip $D:=\{(x,y): 0 < y < \pi/2\}$ (see  Fig. \ref{fig}). Consider the following ``clocks,''
\begin{align}\notag
\phi(t) &= \int_0^t\frac{1}{|\cV(s)|^2}ds,\\
\label{m29.14}
\sigma(t)&=\int_0^{t}\frac{1}{J(s)^2}ds.
\end{align}
It follows from conformal invariance of  two-dimensional Brownian motion (see  \cite[Thm. V (2.5)]{REYO}) that
the process $Z(t)=(Z_1(t),Z_2(t)):=\log \cV({\phi(t)})$ is two-dimensional  Brownian motion jumping from the boundary of $D$ to an appropriate point in $D$ every time it exits $D$.  
Let $R_1, R_2, \dots$ be the times of jumps of $Z$, and let $R_0=0$.

\begin{center}
\begin{figure}[h]
\includegraphics[width=1\linewidth]{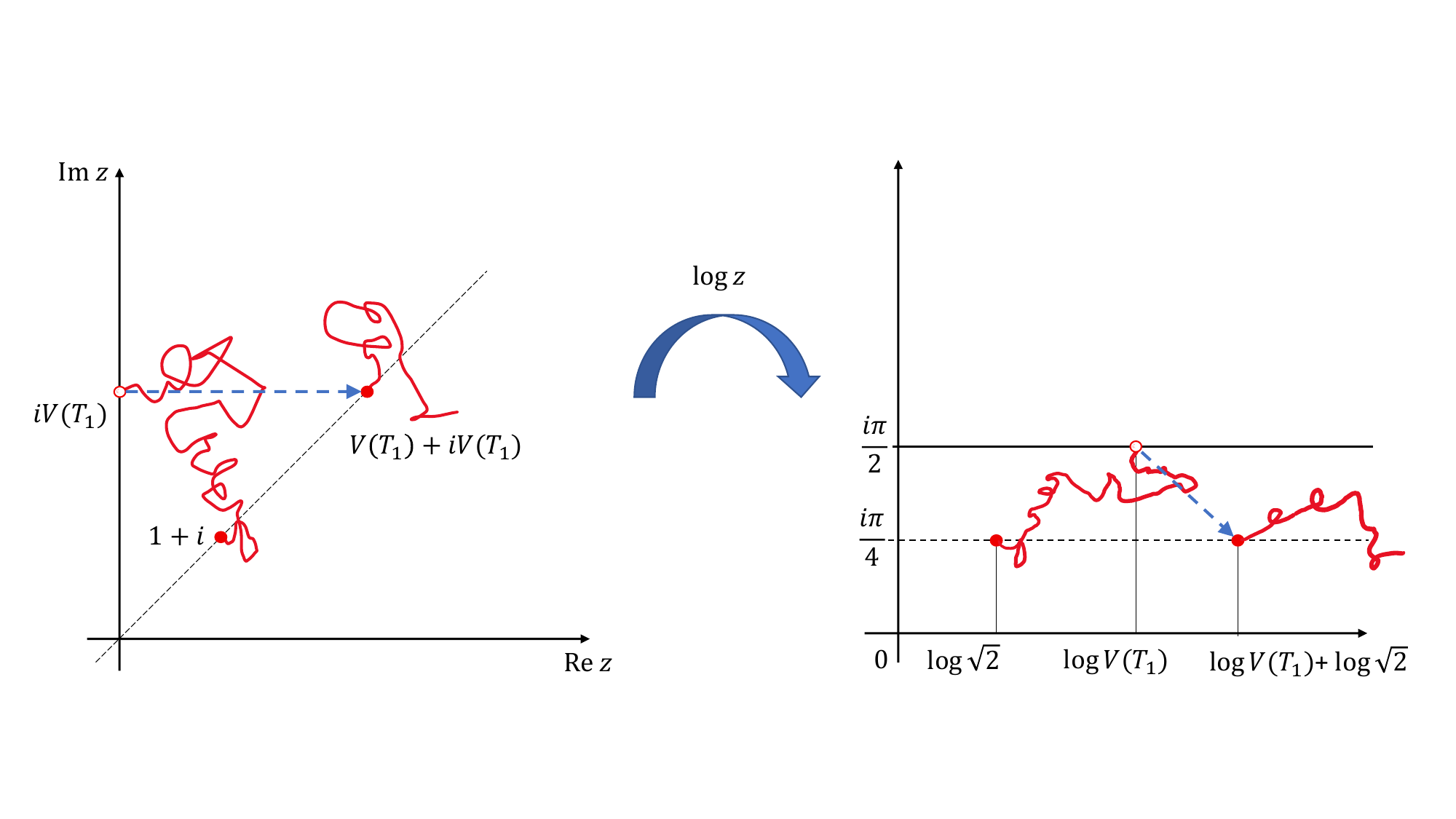}
\caption{The logarithmic transformation of the Fleming-Viot process.}
\label{fig}
\end{figure}
\end{center}

\begin{lemma}
The process
 $\{\log J(T_n), n\geq 0\}$ is a random walk, such that $\log J(T_0) =\log J_0 =0$, and satisfying 
\begin{align}\label{m29.16}
\log J(T_n)=\log J(T_{n-1}) + \log \sqrt{2} + K_n, \qquad n\geq 1,
\end{align}
where $\{K_n,n\geq 1\}$ is an i.i.d. sequence. The distribution of $K_n$ is that of $Z_1(R_1-)$. We also have
\begin{align}\label{m29.1}
 \lim_{t\to \infty} \frac{ \log J(T_n)} n 
= \log\sqrt{2}.
\end{align}
\end{lemma}
\begin{proof}

A jump takes $Z$ from  $(Z_1(R_k-),Z_2(R_k-)) \in \prt D $, i.e., the point at which $Z$ exits $D$, to $\left(\log \sqrt{2} + Z_1(R_k-), \pi/4\right)$. 

Brownian motions driving $Z_1$ and $Z_2$ between jumps are independent of each other. The times $R_k$ are the times when $Z_2$ exits $[0,\pi/2]$.
 Hence, random variables $\{R_k, k\geq 0\}$ are independent of the Brownian motion $B_t$ driving $Z_1(t)$.  The time $R_1$ is the exit time from the interval $[0, \pi/2]$ for a Brownian motion  starting at $\pi/4$ and independent of $B$.
Note that  
$$\log J(T_n) - \log J(T_{n-1})= \log \sqrt{2} + Z_1(R_n-)-Z_1(R_{n-1}).$$
The first two claims of the lemma follow from  independence of $\{R_n,n\geq 1\}$ and $B$, and the fact that $R_n-R_{n-1} \stackrel{d}{=}R_1$.

It is easy to see that $\E|Z_1(R_n-)-Z_1(R_{n-1})| < \infty$. Hence, by symmetry, $\E(Z_1(R_n-)-Z_1(R_{n-1}))=0$. Thus we can use the law of large numbers and \eqref{m29.16} to obtain \eqref{m29.1}.
\end{proof}

\begin{theorem}\label{thm:empLogMinimum}
We have 
$$\lim_{t\to\infty}e^t\P\left(\inf_{n\geq 0} \log J(T_n)<-t\right)=c\in (0,\infty).$$ 
\end{theorem}
\begin{proof}
Since $\{\log J(T_n),n\geq 0\}$ is a random walk, the function
\begin{align*}
t\to \P\left(\inf_{n\geq 0} \log J(T_n)<-t\right)
=\P\left(\sup_{n\geq 0}(- \log J(T_n))>t\right)
\end{align*}
satisfies the Wiener-Hopf equation; see  \cite[point 2, bottom of page 191]{asmussen} for general overview, and \cite[Theorem 3.1]{MR652832}. It follows from these references  
that 
\begin{align}\label{s30.1}
\lim_{t\to\infty}e^{\gamma t}\P\left(\inf_{n\geq 0} \log J(T_n)<-t\right) = \lim_{t\to\infty}e^{\gamma t}\P\left(\sup_{n\geq 0}(- \log J(T_n))>t\right)=c,
\end{align}
where $\gamma$
is the positive solution to the equation 
$$\E\left[e^{\gamma(- \log J(T_1))}\right]=\E[J(T_1)^{-\gamma}]=1.$$
It follows from  \cite[(6.21)]{GK} that
$$\P(J(T_1)\in dy)=\frac{2}{\pi}\left[\frac{1}{(1-y)^2+1}-\frac{1}{(1+y)^2+1}\right].$$
It is not difficult to check that $\E[J(T_1)^{-1}]=1$.
Hence $\gamma = 1$ and, therefore, the theorem follows from \eqref{s30.1}. 
\end{proof}

\begin{proof}[Proof of Theorem \ref{s26.2}]
The theorem follows from Lemma \ref{cor:logBessMin} and Theorem \ref{thm:empLogMinimum}.
\end{proof}

\section{Comparing the spine and 3-dimensional Bessel process}

\begin{lemma}\label{m29.30}
(i) We have
\begin{align}\label{m29.10}
\alpha&:=\E \int_0^{T_1} \frac 1 {J_t^2} dt ==\frac{8 C}{\pi }-\log (2)
\approx 1.63934,
\end{align}
where $C\approx 0.915966$ is the Catalan's constant.

(ii) Random variables
\begin{align}\label{m29.20}
 \int_{T_n}^{T_{n+1}} \frac 1 {J_t^2} dt, \qquad n\geq 0,
\end{align}
are i.i.d.
\end{lemma}

\begin{proof} (i)
We will switch between complex and real notation
and write $z=x+i y$, $|z| = \sqrt{x^2+y^2}$.
Let $A\subset \C$ denote the first quadrant.
The Green function in $A$ with the pole at $1+i $ is given by
\begin{align*}
G(z) &= \frac 1 \pi \log \left|
\frac{z^2 + 2 i}{z^2-2i}
\right| .
\end{align*}
The normalization $1/\pi$ is probabilistic, i.e., the integral of thus normalized Green function is equal to the expected lifetime of Brownian motion starting from $1+i$ and killed upon exiting $A$.

The process $\{J_t, 0\leq t <T_1\}$ has the same distribution as $\{W_2(t), 0 \leq t \leq T_1\}$ conditioned by $\{\tau_1 < \tau_2\}$. Hence we will estimate the expectation in \eqref{m29.10} assuming that $(W_1, W_2)$ is conditioned to exit $A$ through the vertical axis. This process is a Doob's transform, or an $h$-process,
where $h$ is harmonic in $A$ with boundary values 1 on the vertical axis and 0 on the horizontal axis (see \cite[Part 2, Chap. X]{Doob84} or \cite[Ch. 11]{ChungWalsh} for the theory of $h$-processes). The only harmonic function with these boundary values is $h(z) = (2/\pi)\arg(z)=(2/\pi) \arctan(y/x)$. 

The Green function for $(W_1,W_2)$ conditioned by $h$ is
$ G( z)h(z)/h(1+i)= 2 G(z)h(z)$ so
\begin{align}\notag
\E \int_0^{T_1} \frac 1 {J_t^2} dt 
&= \int_A 2 G( z)h(z) \frac 1 {y^2} dz
= \int_A 2 G( x+iy)h(x+iy) \frac 1 {y^2} dx dy\\
&= \int_A 2 G( x+iy)(2/\pi) \arctan(y/x) \frac 1 {y^2} dx dy \label{m29.5}\\
&= \int_A \frac 4 \pi G( x+iy) \arctan(y/x) \frac {x^2+y^2} {y^2}
\frac 1 {x^2+y^2} dx dy \notag \\
&= \int_A \frac 4 \pi G( x+iy) \arctan(y/x) \frac {1+(y/x)^2} {(y/x)^2}
\frac 1 {x^2+y^2} dx dy.\notag
\end{align}
Next we will change the variables. Informally speaking, we will apply the complex function $z \to \log z$. In terms of real coordinates, we take
\begin{align*}
(x, y) &= \left(e^{r}\cos\theta , e^{r}\sin\theta\right),\\
dr d\theta &= \frac 1 {x^2+y^2} dx dy .
\end{align*}
Note that $y/x = \tan \theta$ and $\arctan(y/x) = \theta$.
Let $G_*(r,\theta) = G(x+iy)$ and note that $G_*$ is the Green function in the strip $A_*:=\{(r,\theta): 0< r < \pi/2\}$ with the pole at $(\sqrt{2},\pi/4)$.
We obtain
\begin{align*}
&\int_A \frac 4 \pi G( x+iy) \arctan(y/x) \frac {1+(y/x)^2} {(y/x)^2}
\frac 1 {x^2+y^2} dx dy\\
&=\int_{A_*} \frac 4 \pi G_*( r,\theta) \theta \frac {1+\tan^2\theta} {\tan^2\theta} dr d\theta
=\int_0^{\pi/2} \frac 4 \pi  \theta \frac {1+\tan^2\theta} {\tan^2\theta}\int_{-\infty}^\infty G_*( r,\theta) dr d\theta.
\end{align*}
The function $G_1(\theta):= \int_{-\infty}^\infty G_*( r,\theta) dr$ is the Green function for the one-dimensional Brownian motion starting from $\pi/4$ and killed upon exiting $(0, \pi/2)$. Hence, 
\begin{align*}
G_1(\theta) = 
\begin{cases}
\theta & \text{  for  } 0 < \theta < \pi/4,\\
\pi/2-\theta & \text{  for  } \pi/4< \theta < \pi/2.
\end{cases}
\end{align*}
Note that $G_1(\theta)$ is properly normalized, i.e., $\int_0^{\pi/2}G_1(\theta)d\theta = \pi^2/16$. In other words, the integral is equal to the expected exit time, known to be $\pi^2/16$, from $(0,\pi/2)$ for one-dimensional Brownian motion starting from $\pi/4$.

We obtain
\begin{align*}
&\int_0^{\pi/2} \frac 4 \pi  \theta \frac {1+\tan^2\theta} {\tan^2\theta}\int_{-\infty}^\infty G_*( r,\theta) dr d\theta\\
&=
\int_0^{\pi/4} \frac 4 \pi  \theta \frac {1+\tan^2\theta} {\tan^2\theta}\theta d\theta
+
\int_{\pi/4}^{\pi/2} \frac 4 \pi  \theta \frac {1+\tan^2\theta} {\tan^2\theta}(\pi/2-\theta) d\theta\\
&=\left[\frac{4 C}{\pi }-\frac{\pi }{4}+\log (2)\right]
+ \left[\frac{1}{4} \left(\frac{16 C}{\pi }+\pi -\log (256)\right)\right]
=\frac{8 C}{\pi }-\log (2)
\approx 1.63934,
\end{align*}
where $C\approx 0.915966$ is the Catalan's constant.
The exact values of the integrals were computed using Mathematica. The numerical value was confirmed by numerical calculations (Riemann sum approximation). 

(ii) By Brownian scaling and the strong Markov property applied at $T_n$,
\begin{align*}
\left\{\frac{\cV(T_n +t V^2_1(T_n))}{V_1(T_n)}, t\in[0, (T_{n+1}-T_n)/V_1^2(T_n))\right\}, \quad n\geq 1,
\end{align*}
are i.i.d. For more details see \cite[Lemma 7.10]{ldm2020}. This implies that
\begin{align*}
\left\{\frac{J(T_n +t V^2_1(T_n))}{V_1(T_n)}, t\in[0, (T_{n+1}-T_n)/V_1^2(T_n))\right\}, \quad n\geq 1,
\end{align*}
are i.i.d., and so are
\begin{align*}
 \int_{T_n}^{T_{n+1}} \frac 1 {J_t^2} dt
=  \int_{0}^{T_{n+1}-T_n} \frac 1 {J^2(T_n+t)} dt
=  \int_{0}^{(T_{n+1}-T_n)/V_1^2(T_n)} 
\frac {V_1^2(T_n)} {J^2(T_n+sV_1^2(T_n))} ds.
\end{align*}
\end{proof}

\begin{remark}\label{m29.8}
We have $\log\sqrt{2}\approx 0.346574 <0.81967 \approx \alpha/2$.
\end{remark}

\begin{proof}[Proof of Theorem \ref{s26.1}]

By Lemma \ref{m29.9},  $\{ \log X_{\rho^{-1}(t)}, t\geq 0\}$ has the same distribution as the process $\left\{B_t + \frac 1 2 t, t\geq 0\right\}$. Hence, a.s., 
\begin{align*}
\lim_{t\to \infty} \frac{ \log X_{\rho^{-1}(t)}} t 
= \lim_{t\to \infty} \frac{ B_t + \frac 1 2 t} t = 1/2.
\end{align*}
If $L_n$ is any sequence of positive random variables such that, a.s.,
\begin{align*}
\lim_{n\to\infty} L_n/n = \alpha
\end{align*}
then, a.s.,
\begin{align}\label{m29.11}
\lim_{t\to \infty}  \frac{ \log X_{\rho^{-1}(L_n)}} n 
 = \lim_{t\to \infty} \frac{ \log X_{\rho^{-1}(L_n)}} {L_n/\alpha} 
 = \frac  \alpha 2.
\end{align}

Let $U_n = \sigma(T_n)$, recall definition \eqref{m29.10}, and use Lemma \ref{m29.30} (ii) and the law of large numbers to see that, a.s.,
\begin{align*}
\lim_{n\to\infty} U_n/n = \alpha.
\end{align*}
By \eqref{m29.1}, a.s.,
\begin{align}\label{m29.12}
\lim_{t\to \infty} \frac{ \log J(\sigma^{-1}(U_n))} n 
= \lim_{t\to \infty} \frac{ \log J(T_n)} n 
= \log\sqrt{2}.
\end{align}

Note that the time transformations of processes $\log X$ and $\log J$ are based on ``clocks'' $\rho$ and $\sigma$ defined in an analogous way (see \eqref{m29.13} and \eqref{m29.14}). By Remark \ref{m29.8}, the limits in \eqref{m29.11} and \eqref{m29.12} are a.s. and different so the distributions of $J$ and $X$ are mutually singular.
\end{proof}

\begin{remark}
Theorem \ref{s26.1} compares distributions of processes $J$ and $X$ starting from 1 but it is clear that the limits in \eqref{m29.11} and \eqref{m29.12} do not depend on the initial distributions of the two processes. Hence, the claim in Theorem \ref{s26.1} holds for any initial distributions of $J$ and $X$.
\end{remark}


\bibliographystyle{amsplain}
\bibliography{SpineECP7}


\begin{acks}
We are grateful to Don Marshall and Jan Swart for the most useful advice. 
We thank the referees for finding and correcting a significant error in the first version of this article and for many suggestions for improvement.
\end{acks}


\end{document}